\documentclass{amsart}
\usepackage{amsfonts,amssymb,amsmath,amsthm}
\usepackage{url}
\usepackage{enumerate}
\newtheorem{thrm}{Theorem}[section]
\newtheorem{lem}[thrm]{Lemma}
\newtheorem{prop}[thrm]{Proposition}
\newtheorem{cor}[thrm]{Corollary}
\theoremstyle{definition}

\newtheorem{remark}[thrm]{Remark}
\newcommand{\Gr}{\operatorname{Gr}}

\newcommand{\rk}{\operatorname{rk}}

\newcommand{\gon}{\operatorname{gon}}
\newcommand{\Spec}{\operatorname{Spec}}
\newcommand{\Hom}{\operatorname{Hom}}
\newcommand{\im}{\operatorname{Im}}

\numberwithin{equation}{section}

\author{Ali Bajravani}
\address{Department of Mathematics, Faculty of Basic Sciences, Azarbaijan Shahid Madani University, Tabriz, I. R. Iran.\\
P. O. Box: 53751-71379.\\
}
\email{bajravani@azaruniv.ac.ir}

\begin{document}

\title[A note on the Tangent Cones of ... ]{A note on the Tangent Cones of the scheme of Secant Loci}

\begin{abstract}
The point of this short note concerns with two facts on the scheme of secant loci. The first one is an attempt
to describe the tangent cone of these schemes globally and the second one is a comparison on the dimension of the tangent spaces of various schemes of secant Loci.
\end{abstract}
\maketitle

\section{Introduction and Notations} \label{sect1}
Let $C$ be a smooth projective algebraic curve of genu $g$; $W^0_{g-1}(C)$ its theta divisor and $L\in W^0_{g-1}(C)$ be a multiple point of the theta divisor.
Based on a classical and nice result of Bernhard Riemann, the tangent cone of $W^0_{g-1}(C)$ at $L=\mathcal{O}(D)$ is, set theoretically, the union of the $n$-planes $\Lambda=\langle E \rangle$, where $E$ is the canonical image of a divisor $\acute{E}\in \mid  D \mid$ in the canonical space of $C$ and $n=\deg(D)-h^0(D)$. See \cite[Ch. 6]{ACGH}. G. Kempf generalized this result to the schemes $W^0_d$, when $1\leq d\leq g-1$, (see \cite{K}). Subsequently, Arbarello, Cornalba, Griffiths and Harris used the scheme of linear series, $G^r_d(C)$'s, to give a global description of the tangent cone of the Brill-Noether schemes, $W^r_d$, at their multiple points when $r$ and $d$ ranges in $1\leq 2r\leq d\leq g-1$.

The scheme of secant loci of globally generated line bundles on $C$, being as a generalization of the classical Brill-Noether varieties, was under focus of some authors beginning by M. Coppens in 1990's to recently by M. Aprodu and E. Sernesi. Marc Coppens, M. E. Huibregetse and T. Johnsen, studying the local behavior of these schemes, have given  descriptions of their tangent space and tangent cones at their various points, in terms of their local defining equations.

The first aim of this note is to describe the tangent cone of the scheme of secant loci', globally. In order to do so
the method of \cite{ACGH} in constructing linear series $G^r_d$, goes verbatim to construct analogous schemes
on the varieties of secant divisors. The resulting spaces enjoy a powerful universal property. Based on this property; these schemes, so called "the scheme of divisor series" would be used to obtain a global description for the tangent cones of the scheme of secant divisors.

W. Fulton and etal., established inequalities within the dimension of various Brill-Noether varieties in \cite{F-H-L}. The relations have been extended recently to the varieties of secant loci by M. Aprodu and E. Sernesi in \cite{A-S2}.
Inspired by their results, we report in Theorem \ref{comparision theorem 1} similar inequalities within
$ \dim V^{r}_{d}(\Gamma)$, $\dim  V^{r}_{d}(\Gamma(-x))$, $\dim T_D(V^{r}_d(\Gamma))$, $\dim T_{D+x}(V^{r}_{d+1}(\Gamma))$ and $  \dim T_{D} V^{r}_{d}(\Gamma(-x)),$
where $x$ is a general point of $C$. This is the second aim of this paper. As a corollary to this result, the smoothness of $ V^{r}_{d}(\Gamma)$, when  $ V^{r}_{d}(\Gamma)$ is of expected dimension, implies the same property for $ V^{r}_{d+1}(\Gamma)$ and $ V^{r}_{d}(\Gamma(-x))$.

Assume that $\Gamma$ is a line bundle on a smooth projective algebraic curve $C$ of genus $g$ with $h^0(\Gamma)=s+1$ and $d$ is a positive integer. For an integer $d\geq 2$, consider the diagram
$$\begin{array}{cccccc}
&C\times C_d &\overset{\pi_2}\longrightarrow & C_d\\
\pi_1\!\!\!\!\!\!\!\!\!\!\!\!\!\!\!&\downarrow & &  \\
&C & &  \\
\end{array}$$
 and define the secant bundle of degree $d$;
 $E_{\Gamma}:=(\pi_2)_*(\pi^*\Gamma\otimes \mathcal{O}_{\Delta})$, where $\Delta$ is the universal divisor of degree $d$. The morphism
 $$(\pi_2)_*(\pi^*\Gamma) \overset{\phi_{\Gamma}}\longrightarrow
  (\pi_2)_*(\pi_1^*\Gamma\otimes \mathcal{O}_{\Delta})$$
 is a map of vector bundles of ranks $s+1$ and $d$, respectively.
 For a positive integer $r$, $1\leq r\leq d-1$, the variety of secant loci of $\Gamma$ is the zero scheme of the map $\wedge^{d-r+1}\phi_{\Gamma}$, i.e.
\begin{align}
V^r_d(\Gamma)=Z(\wedge^{d-r+1}\phi_{\Gamma}).
\end{align}
The variety of secant loci of $\Gamma$ migh be described set theoretically as
 $$V^{r}_{d}(\Gamma):=\lbrace D\in C_{d} \mid h^0(\Gamma)-h^0(\Gamma(-D))\leq d-r \rbrace.$$
 See \cite{A-S}, $\cdots$, \cite{J} for more details on the scheme structure of $V_{d}^{r}(\Gamma)$ and some of its geometric properties.

\section{The structure of $\mathcal{V}^{s+1-d+r}_d(\Gamma)$:} \label{ns}
For a closed subscheme $Z\subset X$ defined as the $k$-th degeneracy locus of a morphism of vector bundles
$\gamma: \mathcal{F}\rightarrow \mathcal{G},$
its canonical desingularization, as it is defined in \cite[Page 83-84]{ACGH}, parametrizes couples $(x, W)$ in which $x\in X$ and
$W\in \Gr(n-k, \ker \gamma_x)$,
where $\rk \mathcal{F}=n$ and $\rk \mathcal{G}=m$. Denote
such a desingularization by $\tilde{X}_k(\gamma)$ and
set $\mathcal{V}^{s+1-d+r}_d(\Gamma):=\tilde{X}_{d-r}(\phi_{\Gamma})$.
 Geometrically, the scheme $\mathcal{V}^{s+1-d+r}_d(\Gamma)$ parametrizes couples $(D, \Lambda)$, with $D\in V^{r}_{d}(\Gamma)$ and $\langle D \rangle\subset\Lambda\subset \mathbb{P}(H^0(\Gamma))$ with $\dim \Lambda=d-r-1$. The elements of $\mathcal{V}^{s+1-d+r}_d(\Gamma)$, are called divisor series.

\subsubsection{Families of divisor series:}\label{Families of divisor series:}
A family of divisor series, $\delta^r_d(\Gamma)$, w.r.t. $\Gamma$ parametrized by $S$, is the datum of:\\
\label{condition1}(I) A family $\mathcal{D}$ of degree $d$ divisors on $C$, parametrized by $S$;\\
\label{condition2}(II) A rank $(s+1-d+r)$-vector bundle $\mathcal{T}$, which is a subvector bundle of
$ (\bar{\pi}_2)_*(\bar{\pi}_1^*\Gamma\otimes \mathcal{O}(\mathcal{D})^{\vee}),$
with the property that, for each $s\in S$, the homomorphism
$$\mathcal{T}\otimes k(s)\rightarrow
H^0(  (\bar{\pi}_2)^{-1}(s), [\bar{\pi}_1^*\Gamma\otimes \mathcal{O}(\mathcal{D})^{\vee}]\otimes \mathcal{O}_{(\bar{\pi}_2)^{-1}(s)})$$
  is injective, where $\bar{\pi}_1$ and $\bar{\pi}_2$ are the projections from $C\times S$
to $C$ and $S$, respectively.

Two families $(\mathcal{D}_1, \mathcal{T}_1)$ and $(\mathcal{D}_2, \mathcal{T}_2)$ of $\delta_d^r(\Gamma)$'s on $C$ parametrized by $S$ are said to be equivalent if
$\mathcal{D}_1 = \mathcal{D}_2,$
such that $\mathcal{T}_1$ can be identified via $\mathcal{T}_2$ under this equality.
\subsubsection{The universal family of divisor series:}\label{The universal family of divisor series}
Consider that $\mathcal{V}^{s+1-d+r}_d(\Gamma)$ is a subvariety of the Grassmann bundle
$G(s+1-d+r, (\pi_2)_*\pi_1^*\Gamma)$ over $C_d$. If
$$e:\mathcal{V}_d^{s+1-d+r}(\Gamma)\rightarrow C_d$$
is the restriction of the projection map from $G(s+1-d+r, (\pi_2)_*\pi_1^*\Gamma)$ to $\mathcal{V}^{s+1-d+r}_d(\Gamma)$, then the universal family of $\delta^r_d(\Gamma)$'s on $C$ parametrized by
$\mathcal{V}^{s+1-d+r}_d(\Gamma)$ is $(e^*(\Delta), \mathcal{G})$, where $\mathcal{G}$ is the restriction to $\mathcal{V}^{s+1-d+r}_d(\Gamma)$ of the universal sub-bundle on $G(s+1-d+r, (\pi_2)_*\pi_1^*\Gamma)$ and $\Delta$ is the universal divisor of degree $d$. We denote this family of divisors by $\mathcal{U}\mathcal{V}_d^{s+1-d+r}(\Gamma)$.
\begin{lem}\label{lem1}
Assume that
 $\mathcal{D}$ is a family of degree $d$ divisors on $C$, parametrized by $S$;
and $f:S\rightarrow C_d$ is the unique morphism such that $(f\times id_C)^*(\Delta)=\mathcal{D}$. Then
$$\ker f^*(\phi_{\Gamma})\cong \ker \lbrace(\bar{\pi}_2)_*((f\times id_C)^*(\pi_1^*(\Gamma)))\longrightarrow (\bar{\pi}_2)_*((f\times id_C)^*(\pi_1^*(\Gamma)\otimes \mathcal{O}_{\Delta}))\rbrace$$
\end{lem}
\begin{proof}
\textbf{Claim:} $\pi_1^*(\Gamma)$ and $\pi_1^*(\Gamma)\otimes \mathcal{O}_{\Delta}$ are flat $\mathcal{O}_{C_d}$-modules. Indeed, observe first that $\pi_1^*(\Gamma)$ is flat as $\mathcal{O}_{C_d\times C}$-modules. The flatness of $\mathcal{O}_{C_d\times C}$ as $\mathcal{O}_{C_d}$-modules is a direct consequence of the commutative diagram
$$\begin{array}{cccccc}
&C_d\times C&\overset{\pi_2}\longrightarrow & C\\
\pi_1\!\!\!\!\!\!\!\!\!\!\!\!\!\!\!&\downarrow & & \downarrow \\
&C_d & \longrightarrow&  \Spec(k).\\
\end{array}$$

The flatness of $\pi_1^*(\Gamma)$ and $\pi_1^*(\Gamma)\otimes \mathcal{O}_{\Delta}$ as $\mathcal{O}_{C_d}$-modules, together with Theorem \cite[Thm. 2.6, page 175]{ACGH} applied to the morphism $C_d\times C\rightarrow C_d$ shows that
\begin{align}
f^*(\pi_1^*(\Gamma))\cong (\bar{\pi}_2)_*((f\times id_C)^*(\pi_1^*(\Gamma))
\end{align}
\begin{align}
f^*(\pi_1^*(\Gamma)\otimes \mathcal{O}_{\Delta})\cong
(\bar{\pi}_2)_*((f\times id_C)^*(\pi_1^*(\Gamma)\otimes \mathcal{O}_{\Delta}).
\end{align}
The lemma now is a direct consequence of the commutative diagram of vector bundles on $S$
$$\begin{array}{cccccc}
&f^*(\pi_1^*(\Gamma)) &\overset{f^*(\phi_{\Gamma})}\longrightarrow & f^*(\pi_1^*(\Gamma)\otimes \mathcal{O}_{\Delta})\\
\!\!\!\!\!\!\!\!\!\!\!\!\!\!\!&\downarrow & & \downarrow \\
&(\bar{\pi}_2)_*((f\times id_C)^*(\pi_1^*(\Gamma)) & \longrightarrow&  (\bar{\pi}_2)_*((f\times id_C)^*(\pi_1^*(\Gamma)\otimes \mathcal{O}_{\Delta}).\\
\end{array}$$
\end{proof}
\begin{thrm}\label{thm1}
For any analytic space $S$ and any family $\mathbb{E}$ of divisor series 
 on $C$
parametrized by $S$,
there is a unique morphism from $S$ to $\mathcal{V}^{s+1-d+r}_d(\Gamma)$ such that the pull back of $\mathcal{U}\mathcal{V}_d^{s+1-d+r}(\Gamma)$ is equivalent to $\mathbb{E}$.
\end{thrm}
\begin{proof}
Let $\mathbb{E}=(\mathcal{D}, \mathcal{T})$ be a family of divisor series, $\delta^r_d(\Gamma)$'s, on $C$ parametrized by $S$. The universal property of $\Delta$ asserts that there is a morphism $f:S\rightarrow C_d$ such that
$(f\times id_C)^*(\Delta)=\mathcal{D}$. Condition (II) in \ref{condition2} together with Lemma \ref{lem1} makes it possible to view the vector bundle $\mathcal{T}$ as a vector sub-bundle of $f^*(H^0(\Gamma)\otimes \mathcal{O}_{C_d})$ contained in $f^*(\ker \phi_{\Gamma})$. The universal property of Grassmann bundles implies that the vector bundle $\mathcal{T}$ is the pull back of the universal sub-bundle via a unique section of $G(s+1-d+r, f^*(H^0(\Gamma)\otimes \mathcal{O}_{C_d}))\rightarrow S$. This section factors through the inclusion $\mathcal{V}^{s+1-d+r}_d(\Gamma)\subseteq G(s+1-d+r, H^0(\Gamma)\otimes \mathcal{O}_{C_d})$, since $\mathcal{T}$ is annihilated by $f^*(\phi_{\Gamma})$.
\end{proof}
\subsubsection{The Tangent Space of $\mathcal{V}^{s+1-d+r}_d(\Gamma)$}
 Theorem \ref{thm1} shows that
$T_{\mathbb{E}}(\mathcal{V}^{s+1-d+r}_d(\Gamma))
$ is the set of families of $\delta^r_d(\Gamma)$'s parametrized by
$\Spec(\mathbb{C}[\epsilon])$ reducing to $\mathbb{E}$. A family of this type is called a first order deformation of $\mathbb{E}$.
\begin{thrm}Let $\mathbb{E}=(D, T)\in \mathcal{V}^{s+1-d+r}_d(\Gamma)$.
Then, a first order deformation of $\mathbb{E}$ is in the form  $\mathbb{E}_{\epsilon}=(\mathcal{D}_{\epsilon}, \mathcal{T}_{\epsilon})$, where $\mathcal{D}_{\epsilon}$
is a first order deformation of $D$ and $\mathcal{T}_{\epsilon}\subset
 \Gamma_{\epsilon}(-\mathcal{D}_{\epsilon})$ extends $T$,
 in which $\Gamma_{\epsilon}$ is the trivial first order deformation of $\Gamma$.
\end{thrm}
\begin{proof}
Assume $\mathbb{E}_{\epsilon}=(\mathcal{D}_{\epsilon}, \mathcal{T}_{\epsilon})$ is a family of $\delta^r_d(\Gamma)$'s parametrized by
$\Spec(\mathbb{C}[\epsilon])$. Then, $\mathcal{D}_{\epsilon}$ is a relative degree $d$ divisor on $\Spec(\mathbb{C}[\epsilon])$ and so is a first order deformation of $D$.

 For each $s\in \Spec(\mathbb{C}[\epsilon])$, the vector bundle
$ (\bar{\pi}_2)_*(\bar{\pi}_1^*\Gamma\otimes \mathcal{O}(\mathcal{D}_{\epsilon})^{\vee})$ satisfies
$$\lbrace (\bar{\pi}_2)_*[\bar{\pi}_1^*\Gamma\otimes \mathcal{O}(\mathcal{D}_\epsilon)^{\vee}]\rbrace\otimes k(s)\cong H^0(\Gamma(-\mathcal{D}_s)),$$
where $\mathcal{D}_s$ is the restriction of  $\mathcal{D}_{\epsilon}$ to $\lbrace s\rbrace\times C$. This implies that the vector bundle $ (\bar{\pi}_2)_*(\bar{\pi}_1^*\Gamma\otimes \mathcal{O}(\mathcal{D}_{\epsilon})^{\vee})$ might be viewed as the vector bundle
$ \Gamma_{\epsilon}(-\mathcal{D}_{\epsilon})$,  where $\Gamma_{\epsilon}$ is the trivial first order deformation of $\Gamma$.
So $T$ has to be extended to some sub-vector bundle $\mathcal{T}_{\epsilon}$ of $ \Gamma_{\epsilon}(-\mathcal{D}_{\epsilon})$.
\end{proof}
\begin{prop}\label{Proposition}
 Let $\mathbb{E}=(D, T)\in \mathcal{V}^{s+1-d+r}_d(\Gamma)$ corresponding to a divisor $D\in V^{s+1-d+r}_d(\Gamma)$
and an $(r+1)$-dimensional vector subspace $T$ of $H^0(\Gamma(-D))$. Denote by
$$\mu_{0, T}^{\Gamma}: H^0(D)\otimes T\rightarrow H^0(\Gamma)$$
the restriction of $\mu_{0}^{\Gamma}$ to $H^0(D)\otimes T$. \\
 The tangent space to $\mathcal{V}^{s+1-d+r}_d(\Gamma)$ at $\mathbb{E}$ fits into an exact sequence
$$0\rightarrow \Hom(T, H^0(\Gamma(-D))/T)\rightarrow T_{\mathbb{E}}(\mathcal{V}_d^{s+1-d+r}(\Gamma))\overset{e_*}\longrightarrow T_DC_d .$$
Furthermore, if $\bar{\eta}_T$ is the cup product $H^0(\mathcal{O}_D(D))\otimes T \rightarrow  H^0(\Gamma \otimes \mathcal{O}_D)$, then
$$\im e_*= \lbrace \nu \in H^0(\mathcal{O}_D(D)) \mid \bar{\eta}_T(\nu \otimes T)\subseteq \im \alpha_\Gamma\rbrace.$$
\end{prop}
\begin{proof}
If $ D$ is locally defined by $ (U_{\alpha}, \lbrace f_{\alpha} \rbrace)$
then $(U_{\alpha}, \lbrace g_{\alpha, \beta}:=\frac{f_{\beta}}{f_{\alpha}} \rbrace)$ would be a
transition datum for $\mathcal{O(D)}$. As well,
if $(U_{\alpha}, \lbrace \gamma_{\alpha, \beta} \rbrace)$ determines the line bundle $\Gamma$, then the line bundle $\Gamma(-D)$ would be determined by $(U_{\alpha}, \lbrace \frac{\gamma_{\alpha, \beta}}{g_{\alpha, \beta}} \rbrace)$.

If $D_{\epsilon}$ is a first order deformation of $D$ associated to $\nu \in H^0(\mathcal{O}_D(D))$ and represented by
$(U_{\alpha, \epsilon}, \lbrace \tilde{f}_{\alpha} \rbrace)$, then $(U_{\alpha, \epsilon}, \lbrace \tilde{g}_{\alpha, \beta}:=\frac{\tilde{f}_{\beta}}{\tilde{f}_{\alpha}} \rbrace)$ would be a transition datum for $\mathcal{O}(D_{\epsilon})$,  such that
$$\tilde{g}_{\alpha, \beta}=g_{\alpha, \beta}(1+\epsilon \phi_{\alpha, \beta}) \quad where \quad \phi_{\alpha, \beta}+\phi_{\beta, \gamma}=\phi_{\alpha, \gamma}.$$
Consider that $\phi=\lbrace \phi_{\alpha, \beta}\rbrace \in H^1(\mathcal{O}_C)$ and $\delta (\nu)=\phi$, where $\delta$ is the coboundary map associated to the exact sequence
$$0\rightarrow \mathcal{O}_C\rightarrow \mathcal{O}(D)\rightarrow \mathcal{O}_D(D)\rightarrow 0.$$
 Furthermore, $\Gamma_{\epsilon}(-D_{\epsilon})$ would be represented by $(U_{\alpha, \epsilon}, \lbrace \tilde{\tilde {g}}_{\alpha, \beta} \rbrace)$, such that $\tilde{\tilde {g}}_{\alpha, \beta}=\frac{\tilde{\gamma}_{\alpha, \beta}}{\tilde {g}_{\alpha, \beta}}$, where by triviality of the deformation $\Gamma_{\epsilon}$, one has $\tilde{\gamma}_{\alpha, \beta}=
\gamma_{\alpha, \beta}$. These, imply that
\begin{align}
\tilde{\tilde {g}}_{\alpha, \beta}=
\frac{\gamma_{\alpha, \beta}}{g_{\alpha, \beta}+\epsilon g_{\alpha, \beta} \phi_{\alpha, \beta}}=\frac{\gamma_{\alpha, \beta}}{g_{\alpha, \beta}}[1+\epsilon(-\phi_{\alpha, \beta})].
\end{align}
In order to lift a section $s\in H^0(\Gamma(-D))$ which is represented by $\lbrace s_{\alpha} \rbrace$ with
\begin{align}
s_{\alpha}= \frac{\gamma_{\alpha, \beta}}{g_{\alpha, \beta}}s_{\beta}\quad
 on \quad U_{\alpha}\cap U_{\beta},
 \end{align}
 to  a section $\tilde{s}$ of $\Gamma_{\epsilon}(-D_{\epsilon})$ it is necessary and sufficient for $\tilde{s}$ to be represented by $\tilde{s}_{\alpha}$ with $\tilde{s}_{\alpha}=\frac{\tilde{\gamma}_{\alpha, \beta}}{\tilde {g}_{\alpha, \beta}}\tilde{s}_{\beta}$ on $U_{\alpha,\epsilon}\cap U_{\beta,\epsilon}$ such that one has locally
\begin{align}\label{10}
\tilde{s}_{\alpha}=s_{\alpha}+\epsilon\acute{s}_{\alpha}.
\end{align}
Setting $\bar{g}_{\alpha, \beta}:=\frac{\gamma_{\alpha, \beta}}{g_{\alpha, \beta}}$, the equation (\ref{10}) is equivalent to say that 
\begin{align}
 \quad s_{\alpha}=\bar{g}_{\alpha, \beta}\cdot s_{\beta} \quad on \quad U_{\alpha}\cap U_{\beta},
\end{align}
\begin{align}\label{11}
\bar{g}_{\alpha,\beta}. \acute{s}_{\beta}-\acute{s}_{\alpha}=s_{\alpha}.\phi_{\alpha,\beta}, \quad on \quad U_{\alpha}\cap U_{\beta}.
\end{align}
It is an immediate computation to see that the right-hand side in (\ref{11}) is a cocycle representing the cup-product $\phi . s\in H^1(\Gamma(-D))$ under the natural pairing
$$H^1(\mathcal{O}_C)\otimes H^0(\Gamma(-D))\rightarrow H^1(\Gamma(-D)).$$
 Consider the commutative diagram of vector spaces
 $$\begin{array}{cccccc}
&H^0(\mathcal{O}_D(D))\otimes H^0(\Gamma(-D)) &\overset{\delta \otimes 1}\longrightarrow & H^1(\mathcal{O}_C)\otimes H^0(\Gamma(-D))&\\
\bar{\eta} \!\!\!\!\!\!\! \!\!\!\!\!\!\!  \!\!\!\!\!\!\! \!\!\!\!\!\!\! \!\!\!\!\!\!\! \!\!\!\!\!\!\! \!\!\!\!\!\!\! \!\!\!\!\!\!\! \!\!\!\!\!\!\! \!\!\!\!\!\!\! &\downarrow & & \downarrow & \!\!\!\!\!\!\! \!\!\!\!\!\!\!  \!\!\!\!\!\!\! \!\!\!\!\!\!\!  \!\!\!\!\!\!\! \!\!\!\!\!\!\!  \!\!\!\!\!\!\! \eta\\
&H^0(\Gamma\otimes \mathcal{O}_D) & \overset{\bar{\delta}}\longrightarrow&  H^1(\Gamma(-D)).&\\
\end{array}$$
and observe that $[\eta \circ (\delta \otimes 1)](\nu\otimes s)=0$ in $H^1(\Gamma(-D))$. So the commutativity of diagram implies $\bar{\eta}(\nu\otimes s)\in \ker (\bar{\delta})=\im \alpha_{\Gamma}$. This finishes the proof.
\end{proof}
\begin{thrm}
For $D\in V^r_d(\Gamma)$, consider the set $\bar{I}\subseteq H^0(\mathcal{O}_D(D))\times \Gr(s+1-d+r, H^0(\Gamma(-D)))$
defined by
$$ \bar{I}:=\lbrace (\nu, T)\mid
 \bar{\eta}_T(\nu \otimes T)\subseteq \im \alpha_\Gamma
 \rbrace.
$$
Then, the tangent cone of $V^r_d(\Gamma)$ at $D$ coincides on $I:=\pi_1(\bar{I})$ set theoretically, where $\pi_1$ is the first projection on $H^0(\mathcal{O}_D(D))$.
\end{thrm}
\begin{proof}
An application of the Corollary in page 66 of \cite{ACGH} together with Proposition \ref{Proposition} shows
$$\mathcal{T}_D(V^r_d(\Gamma))=I,$$
set theoretically.
\end{proof}
\begin{remark}\label{remark1}
Assume that $\{\gamma_1, \cdots, \gamma_{s+1}\}$ is a basis for $H^0(\Gamma)$. The Brill-Noether matrix $(\gamma_i(p_j))_{i,j}$ defines the structure of $V^r_d(\Gamma)$ locally. This allows one, to interpret $H^0(\Gamma\otimes \mathcal{O}_D)^*$ as the tangent space of $C_d$ at $D$, which is the same as identifying $H^0(\mathcal{O}_D(D))$ with $H^0(\Gamma\otimes \mathcal{O}_D)^*$. If $\{ \frac{1}{z_i}\}_i$ is a basis for $H^0(\mathcal{O}_D(D))$, then such an identification might be given explicitly as
\begin{align}
 \Theta: \frac{1}{z_i} \in H^0(\mathcal{O}_D(D))&\mapsto
(\frac{\gamma_i-p_1}{z_i}(p_1), \cdots, \frac{\gamma_i-p_d}{z_i}(p_d))^* \in H^0(\Gamma\otimes\mathcal{O}_D)^*,
\end{align}
where $z_i$ is a local coordinate around $p_i$ and for $v$ in a vector space $V$, we denote by $v^*\in V^*$ the linear map by $v^*(\lambda v)=\lambda$ and zero, otherwise.
\end{remark}

In order to obtain Theorem \ref{thm 3}, we make the following hypothesis
\vspace{.2cm}

\noindent \textbf{Hypothesis A:}\label{assumption}
Consider the set $\bar{J}\subseteq H^0(\Gamma)^*\times \Gr(s+1-d+r, H^0(\Gamma(-D)))$, defined by
$$\bar{J}:=\lbrace (\gamma, T)\mid \gamma \perp\mu_0^\Gamma(H^0(D)\otimes T) \rbrace,$$
and assume that the map $\Theta$ is such that setting
 $J:=\bar{\pi}_1(\bar{J})$ the set $(\alpha_\Gamma^*)^{-1}(J)$ coincides on $I$, where
 $$\alpha_\Gamma^*:H^0(\Gamma\otimes \mathcal{O}_D)^*\rightarrow H^0(\Gamma)^*,$$
is the dual of $\alpha_{\Gamma}$ and $\bar{\pi}_1$ is the projection on $H^0(\Gamma)^*$.

Consider the set $\bar{J}\subseteq H^0(\Gamma)^*\times \Gr(s+1-d+r, H^0(\Gamma(-D)))$, defined by
$$\bar{J}:=\lbrace (\gamma, T)\mid \gamma \perp\mu_0^\Gamma(H^0(D)\otimes T) \rbrace,$$
and assume that the map $\Theta$ is such that setting
 $J:=\bar{\pi}_1(\bar{J})$ the set $(\alpha_\Gamma^*)^{-1}(J)$ coincides on $I$, where
 $$\alpha_\Gamma^*:H^0(\Gamma\otimes \mathcal{O}_D)^*\rightarrow H^0(\Gamma)^*,$$
is the dual of $\alpha_{\Gamma}$ and $\bar{\pi}_1$ is the projection on $H^0(\Gamma)^*$.
\begin{thrm}\label{thm 3}
Together with Hypothesis A, assume that for each $T\in \Gr (s+1-d+r, H^0(\Gamma(-D)))$, the
map $\eta_T: H^0(D)\otimes T\rightarrow H^0(\Gamma)$, is injective and $h^0(\Gamma(-D))< h^0(D)+s+1-d+r$. Assume moreover that the scheme $\mathcal{V}^r_d(\Gamma)$ is of dim $=\exp. dim V^r_d(\Gamma)$ in a neighborhood $e^{-1}(D)$. Then $\mathcal{T}_D(V^r_d(\Gamma))$, the tangent cone of $V^r_d(\Gamma)$ at $D$, is generically a $\mathbb{C}^r$-bundle on a reduced, normal and Cohen-Macauley variety $J\subset H^0(\Gamma)^*$.
\end{thrm}
\begin{proof}
Note that the scheme structures on $I$ and $\mathcal{T}_D(V^r_d(\Gamma))$ are compatible with the  structures of the schemes fitted in the commutative diagram
 $$\begin{array}{cccccc}
&\bar{I} &\overset{\zeta \otimes 1}\longrightarrow & \bar{J}&\\
\bar{\eta} \!\!\!\!\!\!\! \!\!\!\!\!\!\!  \!\!\!\!\!\!\! \!\!\!\!\!\!\! \!\!\!\!\!\!\! \!\!\!\!\!\!\! \!\!\!\!\!\!\! \!\!\!\!\!\!\! \!\!\!\!\!\!\! \!\!\!\!\!\!\! &\downarrow & & \downarrow & \!\!\!\!\!\!\! \!\!\!\!\!\!\!  \!\!\!\!\!\!\! \!\!\!\!\!\!\!  \!\!\!\!\!\!\! \!\!\!\!\!\!\!  \!\!\!\!\!\!\! \eta\\
&I & \overset{\alpha_\Gamma^*}\longrightarrow& J,&\\
\end{array}$$
induced from
 $$\begin{array}{cccccc}
&H^0(\mathcal{O}_D(D))\times H^0(\Gamma(-D)) &\overset{\zeta \otimes 1}\longrightarrow & H^0(\Gamma)^*\times H^0(\Gamma(-D))&\\
\bar{\eta} \!\!\!\!\!\!\! \!\!\!\!\!\!\!  \!\!\!\!\!\!\! \!\!\!\!\!\!\! \!\!\!\!\!\!\! \!\!\!\!\!\!\! \!\!\!\!\!\!\! \!\!\!\!\!\!\! \!\!\!\!\!\!\! \!\!\!\!\!\!\! &\downarrow & & \downarrow & \!\!\!\!\!\!\! \!\!\!\!\!\!\!  \!\!\!\!\!\!\! \!\!\!\!\!\!\!  \!\!\!\!\!\!\! \!\!\!\!\!\!\!  \!\!\!\!\!\!\! \eta\\
&[H^0(\Gamma\otimes \mathcal{O}_D)]^* & \overset{\alpha_\Gamma^*}\longrightarrow&  H^0(\Gamma)^*,&\\
\end{array}$$
where we are denoting $\alpha_\Gamma^*\circ \Theta$ by $\zeta$.
This shows $\mathcal{T}_D(V^r_d(\Gamma))=I$, scheme theoretically as well.

In order to finish the proof of theorem, denoting by $\lambda$ the restriction of $\zeta^*$ to $I$, it is enough to prove $\lambda(I)=J$. To do so,
the scheme $\bar{J}$, being a vector bundle on $\Gr(s+1-d+r, H^0(\Gamma(-D)))$, is irreducible, implying the irreducibility of $J$. For a similar reason $I$ comes to be irreducible.

The Lemma in page 242 of \cite{ACGH} applied to the injectivity assumption, implies that $J$ has the claimed properties.

Fianlly, a dimension computation indicates that $\lambda(I)$ can not include $J$ strictly, verifying $\lambda(I)=J$. Meanwhile, the computation indicates that for a general $j\in J$, the dimension of the fiber of $\lambda$ at $j$ equals $r$.
\end{proof}
\begin{remark}
The canonical bundle satisfies in the assumption \ref{assumption}, so the tangent cone of $C^r_d$ at a point $D\in C^r_d$ is generically a $P^r$-bundle on the tangent cone of $W^r_d$ at $L=\mathcal{O}(D)\in W^r_d$.
\end{remark}
\section{A Tangent Space Comparision}
The tangent space to $V^{r}_d(\Gamma)$ at a point $D\in V^{r}_d(\Gamma)\setminus V^{r+1}_d(\Gamma)$ has been described by M. Coppens in \cite[Thm. 0.3]{M. Cop.} as;
$$T_D(V^{r}_d(\Gamma))=
\bigcap _{\xi\in H^0(\Gamma(-D))}\lbrace
 \beta_{\xi}^{-1}(\im(\phi_{\Gamma}^D))  \rbrace,$$
 where for $\xi\in H^0(\Gamma(-D))$ the map $\beta_{\xi}:H^0(\mathcal{O}_D(D))\rightarrow H^0(\Gamma\otimes \mathcal{O}_D)$ is defined by $\nu \mapsto \nu\otimes \xi$ and $\phi_{\Gamma}^D$ is the morphism induced by $\phi_{\Gamma}$ at the point $D$.
 This interpretation describes the tangent space as a subspace of
 the space of first order deformations of $D$, where $D$ is considered as a closed subscheme of $C$.
\begin{thrm}\label{comparision theorem 1}
 Let $x\in C$ be a general point such that $D\in V^{r}_{d}(\Gamma)\setminus V^{r+1}_{d}(\Gamma),  D+x\in V^{r}_{d+1}(\Gamma)\setminus V^{r+1}_{d+1}(\Gamma)$ and $D\in V^{r}_{d}(\Gamma(-x))\setminus V^{r+1}_{d}(\Gamma(-x))$. Then
$$
\begin{array}{cccc}
(a)& \dim T_D(V^{r}_d(\Gamma))&\geq &\dim T_{D+x}(V^{r}_{d+1}(\Gamma))-(r+1),\\
(b) &\dim T_D(V^{r}_{d}(\Gamma))&\geq &\!\!\!\!\!\!\!\!\!\!\!\!\!\!\!\!\!\! \dim T_{D} V^{r}_{d}(\Gamma(-x))-r\\
(c) &\dim V^{r}_{d}(\Gamma)\!\!\!\!\!\!\!\!\!\!\!\!&\geq &\!\!\!\!\!\!\!\!\!\!\!\!\!\!\!\!\!\!\!\!\!\!\!\!\! \dim  V^{r}_{d}(\Gamma(-x))-r.
\end{array}
$$
\end{thrm}
\begin{proof}
(a) We interpret $H^0(\mathcal{O}_D(D))$ as a subspace of $H^0(\mathcal{O}_{D+x}(D+x))$ and $H^0(\Gamma\otimes \mathcal{O}_D)$ as a subspace of $H^0(\Gamma\otimes \mathcal{O}_{D+x})$.
Using these interpretations we obtain a commutative diagram:
$$\begin{array}{cccccc}
&H^0(\mathcal{O}_D(D)) &\!\!\!\!\!\!\!\!\!\!\!\!\!\!\!\!\!\!\!\!\!\!\!\!\!\!\!\!\overset{\beta_\xi}\longrightarrow H^0(\Gamma\otimes \mathcal{O}_D)& \\
i_1\!\!\!\!\!\!\!\!\!\!\!\!\!\!\!\!\!\!\!\!\!\!\!\!\!\!\!\!\!\!\!\!\!\!\!\!\!\!\!\!\!&\downarrow &  \downarrow &\!\!\!\!\!\!\!\!\!\!\!\!\!\!\!\!\!\!\!\!\!\!\!\!\!\!\!\!\!\!\!\!\!\!\!\!\!\!\!i_2\\
&H^0(\mathcal{O}_{D+x}(D+x)) &\overset{\bar{\beta}_\xi}\longrightarrow H^0(\Gamma\otimes \mathcal{O}_{D+x}),&
\end{array}$$
in which $\beta_\xi$ coincides on the  restriction of $\bar{\beta}_\xi$ to $H^0(\mathcal{O}_D(D))$.
Set
$$H^0(\Gamma(-D))=H^0(\Gamma(-D-x))\oplus \langle \gamma \rangle,$$
and observe that if $\lbrace\gamma_1, \cdots, \gamma_t\rbrace$ is a basis for $H^0(\Gamma(-D-x))$, then
$$
\begin{array}{ccccc}
T_DV^r_d(\Gamma)&=&[\bigcap_{i=1}^{i=t}
 \beta_{\gamma_i}^{-1}(\im(\phi_{\Gamma}^D))]\cap \beta_{\gamma}^{-1}(\im(\phi_{\Gamma}^D))&&\\
 &=&[\bigcap_{i=1}^{i=t}
 \bar{\beta}_{\gamma_i}^{-1}(\im(\bar{\phi}_{\Gamma}^{D+x}))\cap H^0(\mathcal{O}_D(D))]\cap \beta_{\gamma}^{-1}(\im(\phi_{\Gamma}^D)).
\end{array}
$$
This implies that $T_DV^r_d(\Gamma)=[T_{D+x}V^r_{d+1}(\Gamma)]\cap H^0(\mathcal{O}_D(D)) \cap  \beta_{\gamma}^{-1}(\im(\phi_{\Gamma}^D))$. Observe furthermore that
$$H^0(\mathcal{O}_D(D))=[T_{D+x}V^r_{d+1}(\Gamma)\cap H^0(\mathcal{O}_D(D))]+ \beta_{\gamma}^{-1}(\im(\phi_{\Gamma}^D)),$$
by which we obtain
$$\dim T_DV^r_d(\Gamma)=\dim T_{D+x}V^r_{d+1}(\Gamma)-1+\dim \beta_{\gamma}^{-1}(\im(\phi_{\Gamma}^D)) -d.$$
The assertion would be a direct consequence of the inequality
 \begin{align}\label{inequality}
 \beta_{\gamma}^{-1}(\im(\phi_{\Gamma}^D))\geq d-r=\dim \im(\phi^D_{\Gamma}).
 \end{align}
In order to prove the inequality (\ref{inequality}), set $V=\beta_{\gamma}^{-1}(\im(\phi_{\Gamma}^D))$  and observe that
$$ \begin{array}{ccc}
\dim V=\dim [\ker \beta_{\gamma}\cap V]+\dim [\im \beta_{\gamma}\cap \im \phi^D_{\Gamma}]
=\dim \ker \beta_{\gamma}+\dim [\im \beta_{\gamma}\cap \im \phi^D_{\Gamma}] \\
=d-\dim \im \beta_{\gamma}+\dim [\im \beta_{\gamma}\cap \im \phi^D_{\Gamma}]=
d-(\dim \im \beta_{\gamma}-\dim [\im \beta_{\gamma}\cap \im \phi^D_{\Gamma}]).
\end{array}$$
The assertion is now immediate by
$$\dim \im \beta_{\gamma}-\dim [\im \beta_{\gamma}\cap \im \phi^D_{\Gamma}]=\dim (\frac{\im \beta_{\gamma}+ \im \phi^D_{\Gamma}}{ \im \phi^D_{\Gamma}})\leq \dim \frac{H^0(\Gamma\otimes \mathcal{O}_D)}{ \im \phi^D_{\Gamma}}=r.$$
(b) For $\xi \in H^0(\Gamma-x-D)$ we are in the situation of the following diagram:
\unitlength .500mm
\linethickness{0.5pt}
\ifx\plotpoint\undefined\newsavebox{\plotpoint}\fi
\begin{center}
\begin{picture}(148.5,95)(0,35)
\put(5,77){\makebox(0,0)[cc]{$H^0(\mathcal{O}_D(D))$}}
\put(45,115){\makebox(0,0)[cc]{$H^0(\Gamma(-x)\otimes \mathcal{O}_D)$}}
\put(51,46){\makebox(0,0)[cc]{$H^0(\Gamma\otimes \mathcal{O}_D)$}}
\put(140,115){\makebox(0,0)[cc]{$H^0(\Gamma(-x))$}}
\put(130,45){\makebox(0,0)[cc]{$H^0(\Gamma)$}}
\put(51.5,109){\vector(1,1){.07}}
\multiput(27.25,84.75)(.0337078652,.0337078652){700}{\line(0,1){.0337078652}}
\put(51.5,49.3){\vector(1,-1){.07}}
\multiput(26.5,74.5)(.0337273992,-.0337273992){700}{\line(0,-1){.0337273992}}
\put(73,115.5){\vector(-1,0){.07}}
\put(75,115.5){\line(1,0){45}}
\put(72,45.25){\vector(-1,0){.07}}
\put(75,45.25){\line(1,0){40}}
\put(129.5,51.25){\vector(0,-1){.07}}
\multiput(129.75,105.5)(-.03125,-7.40625){7}{\line(0,-1){7.40625}}
\put(55.5,51.25){\vector(0,-1){.07}}
\multiput(55.5,105.5)(-.03125,-7.40625){7}{\line(0,-1){7.40625}}
\put(35,100){\makebox(0,0)[cc]{$\acute{\beta}_\xi$}}
\put(90,121){\makebox(0,0)[cc]{$\phi^D_\Gamma(-x)$}}
\put(35,57.75){\makebox(0,0)[cc]{$\beta_\xi$}}
\put(90,36){\makebox(0,0)[cc]{$\phi_\Gamma^D$}}
\put(138,80){\makebox(0,0)[cc]{$i_2$}}
\put(50,80){\makebox(0,0)[cc]{$i_1$}}
\end{picture}
\end{center}
 where $i_1$ and $i_2$ are inclusions. It is easy to see that $\im \phi_\Gamma^D=\im \phi^D_\Gamma(-x)$, by which we obtain $\beta^{-1}_\xi(\im \phi_\Gamma^D)= \acute{\beta}^{-1}_\xi(\im\phi^D_{\Gamma(-x)})$. This implies that, with $\gamma$ as in the proof of the previous case, we have
$$T_DV^r_d(\Gamma)=T_DV^r_d(\Gamma(-x))\cap \beta_\gamma ^{-1}(\im \phi_\Gamma^D).$$
The rest of the proof goes verbatim as in part (a).

(c) We might assume that $V^r_d(\Gamma(-x))$ and $V^r_{d+1}(\Gamma)$ are irreducible. A general $x\in C$ can not stand in the support of all divisors $D\in V^r_{d+1}(\Gamma)$. Otherwise; if for any $E+x\in V^r_{d+1}(\Gamma)$ the divisor $E$ belongs to $V^r_{d}(\Gamma)$, then $\dim V^r_{d+1}(\Gamma)=\dim V^r_{d}(\Gamma)$ which is impossible. If the divisor $E$ belongs to $V^r_{d}(\Gamma(-x))\setminus V^r_{d}(\Gamma)$ for some $E+x\in V^r_{d+1}(\Gamma)$, then one has $h^0(E)=0$ by \cite[Lemma 3.3]{A. B1} which once again is impossible.

For an open subset $U\subset C$, from the equality
$V^r_{d+1}(\Gamma)=\overline{\cup_{p\in U}\{p+V^r_{d}(\Gamma(-p))\}}$
we obtain $\dim V^r_{d}(\Gamma(-x))=\dim V^r_{d+1}(\Gamma)-1$, for general $x\in C$. Indeed for such $x$ the equality
$\dim V^r_{d}(\Gamma(-x))=\dim V^r_{d+1}(\Gamma)$ implies that any $D\in V^r_{d+1}(\Gamma)$
contain $x$ in its support, which is absurd by what we just proved.
This by \cite[Thm. 4.1]{A-S2}, implies the assertion.
\end{proof}
\begin{cor}\label{corollary 2}
If $V^{r}_{d}(\Gamma)$ is smooth at $D\in V^{r}_{d}(\Gamma)$ and of expected dimension,
then for general $x\in C$, $V^r_d(\Gamma(-x))$ would be of expected dimension and smooth at $D\in V^r_d(\Gamma(-x))$. The same conclusion is valid for $V^{r}_{d+1}(\Gamma)$, i.e. it would be of expected dimension and smooth at $D+x \in V^{r}_{d+1}(\Gamma)$.
\end{cor}
\begin{proof}
Based on the inequality $\dim V^r_d(\Gamma(-x))\geq d-r(s-d+r)$, the assertion on the dimension of $V^r_d(\Gamma(-x))$ is a consequence of Theorem \ref{comparision theorem 1}(c), by which part (b) of the same theorem verifies the smoothness assertion for $V^r_d(\Gamma(-x))$ at $\in V^r_d(\Gamma(-x))$. The same argument goes verbatim for smoothness of $V^r_{d+1}(\Gamma)$
at $D+x$. Meanwhile, the assertion on its dimension is concluded by \cite[Thm. 4.1]{A-S2} 
\end{proof}
\begin{cor}\label{corollary 1}
 Assume that $\Gamma(-x)$ turns to be very ample for general $x\in C$. If non-empty, then $\dim V^1_{s-1}(\Gamma)$ is $(s-4)$-dimensional.
\end{cor}
\begin{proof}
Theorem \ref{comparision theorem 1}(c) together with \cite[Lemma 4.4]{A. B1}
implies the corollary.
\end{proof}
\begin{remark}
(a) Theorem \ref{comparision theorem 1} implies Aprodu-Sernesi's result for reduced
$V^r_d(\Gamma)$'s.

\noindent(b) Corollary \ref{corollary 1} is invalid without the very ampleness assumption on $\Gamma(-x)$, see \cite[Ch. VIII. Exe. F]{ACGH}.

\noindent (c) The equality $\gon(C)=[\frac{g+1}{2}]$ is hold for general curves by which one can prove that for general $x_1,\cdots, x_k$ ($1\leq k\leq [\frac{g-1}{2}]$) the line bundle
$K(-x_1-\cdots -x_k)$ turns to be very ample on general curves. Using this fact together with Theorem \ref{comparision theorem 1}(c) one can reprove $\dim C^1_d=2d-g+1$.

\noindent (d) The special case $r=1$ from Theorem \ref{comparision theorem 1}(c) has been proved and was used to prove the main theorem, Theorem 1.3, in \cite{A. B 2}.
\end{remark}

\end{document}